\documentclass[12pt,dvips]{article}
\usepackage[T2A]{fontenc}
\usepackage[utf8]{inputenc}
\usepackage{amssymb,amsmath,amsfonts,amsthm,amscd,latexsym,verbatim,indentfirst,stmaryrd,geometry}
\newtheorem{theorem}{Theorem}[section]
\theoremstyle{definition}

\newtheorem{corollary}[theorem]{Corollary}
\newtheorem{lemma}[theorem]{Lemma}
\newtheorem{proposition}[theorem]{Proposition}
\newtheorem{question}[theorem]{Question}

\newtheorem{definition}[theorem]{Definition}
\newtheorem{example}[theorem]{Example}
\newtheorem{remark}[theorem]{Remark}
\numberwithin{equation}{section}
\def\id{\mathrm{id}}

\def\op{\mathrm{op}}

\def\fun{\mathrm{fun}}
\def\Fun{\mathrm{Fun}}

\def\Aut{\mathrm{Aut}}

\def\Imm{\mathrm{Im}\,}

\begin{document}
\sloppy

\hfill{17B38, 20D08 (MSC2020)}

\begin{center}
{\Large
Rota---Baxter operators on groups}

\smallskip

Valeriy G. Bardakov, Vsevolod Gubarev
\end{center}

\begin{abstract}
Theory of Rota---Baxter operators on rings and algebras
has been developed since 1960.
Recently, L. Guo, H. Lang, Y. Sheng [arXiv:2009.03492]
have defined the notion of Rota---Baxter operator on a group.
We provide some general constructions of Rota---Baxter operators on a group.
Given a map on a group, we study its extensions to a Rota---Baxter operator.
We state the connection between Rota---Baxter operators on a group and
Rota---Baxter operators on an associated Lie ring.
We describe Rota---Baxter operators on sporadic simple groups.

{\it Keywords}:
Rota---Baxter operator, Rota---Baxter group, simple group, sporadic group, factorization.
\end{abstract}

\section{Introduction}

Rota---Baxter operators for commutative algebras firstly appear in the paper of G.~Baxter~\cite{Baxter} in 1960.
Since that time the theory of Rota---Baxter operators has been developed by different authors including
J.-C. Rota, L. Guo, C. Bai, and K. Ebrahimi-Fard, see details in the monograph~\cite{GuoMonograph}.
Maybe, the most important role of Rota---Baxter operators is their connections to
Yang---Baxter equation~\cite{BelaDrin82,Semenov83}, Loday algebras~\cite{BBGN2011,Embedding},
double Poisson algebras~\cite{DoubleLie,Schedler} and others.

In September 2020~\cite{Guo2020}, L. Guo, H. Lang, Y. Sheng defined the notion of Rota---Baxter operator
(RB-operator, short) on a group. A map $B\colon G\to G$ is called a~Rota---Baxter operator (of weight~1) if
$B(g)B(h) = B( g B(g) h B(g)^{-1} )$ for all $g,h\in G$.
A~group~$G$ with a Rota---Baxter operator $B$ is called a~Rota---Baxter group.

The main motivation of~\cite{Guo2020} to introduce this new notion is following.
Given a Rota---Baxter Lie group $(G, B)$, then the tangent map of $B$
at the identity is a Rota---Baxter operator of weight 1 on the Lie algebra of the Lie group~$G$.
If a crossed homomorphism~$\phi$ from a group~$G$ to itself is invertible, then $\phi^{-1}$
is an RB-operator on~$G$~\cite{Guo2020}.
Some basic examples and properties of Rota---Baxter groups
were also stated in the pioneer work~\cite{Guo2020}.

In November 2020~\cite{Goncharov2020}, M. Goncharov introduced the notion of
a~Rota---Baxter operator on a~cocommutative Hopf algebra
as a combination of both Rota---Baxter operators on a~group and on a~Lie algebra.
In particular, every such Rota---Baxter operator on the Hopf algebra $\Bbbk[G]$
arises as a linear extension of an RB-operator defined on a group~$G$.

The {\bf main goal} of the current work is to give different constructions
of Rota---Baxter operators on a~group. We provide constructions via
factorizations, homomorphisms, projections, direct products, and ones defined by exact formulas (see~\S4--6).
In particular, we state in Proposition~\ref{prop:RB-Hom}
that any homomorphism from a group to its abelian subgroup
is a~Rota---Baxter operator.
In~\cite{AnBai,Braga,SumOfFields}, RB-operators on $\Bbbk^n$,
a~sum of $n$~copies of a~field~$\Bbbk$ considered as a commutative $\Bbbk$-algebra, were studied.
In Theorem~\ref{theo:directProduct}
we construct by an RB-operator~$R$ on $\Bbbk^n$
an RB-operator~$B$ on $G^n = G\times G\times \ldots \times G$ for a given group~$G$.

We study extensions of maps on a given group to an RB-operator and extensions of Rota---Baxter groups (see~\S7--8).
We find the necessary (Lemma~\ref{lem:RB-extCond}) and sufficient conditions (Theorem~\ref{thm:RBextension})
under which a~map~$\beta$ on a~group~$G$ can be extended to an RB-operator on~$G$.

Let $G$ be a group, it is well-known that by its lower central series $G_n$, $n=1,2,\ldots$,
one can associate Lie ring
$L(G) = \bigoplus_{n\geq1}G_n/G_{n+1}$.
We prove that by any RB-operator~$B$ on $G$ such that
$B(G_n)\subset G_n$ for all $n\geq1$,
we may define the RB-operator of weight~1 on $L(G)$ (Proposition~\ref{prop:LieRing}).

It is natural to ask about complete description of RB-operators on a~given group~$G$.
Naturally, one should do such classification up to action of automorphisms of~$G$ (see Lemma~\ref{lem:Aut})
and up to involution~$\tilde{}$ (see Lemma~\ref{lem:tilde}).
We make the following step in the direction of classification of RB-operators on finite simple groups:
we describe RB-operators on all sporadic simple groups.

Let us give short {\bf outline} of the work.
In~\S2, we give required preliminaries on Rota---Baxter operators on groups.
In~\S3, we consider the induced on a group~$G$ by an RB-operator~$B$ new group product
$g\circ h = gB(g)hB(g)^{-1}$. The group $G_B = \langle G,\circ\rangle$ was defined
in~\cite{Guo2020}.

In~\S4, we study constructions of RB-operators appeared due to exact factorizations
(double and triple) of groups
and properties of such constructions.

In~\S5, we give constructions of RB-operators
involved homomorphisms, projections and some exact formulas.
We find RB-operators on $n$-abelian groups (Proposition~\ref{prop:n-abelian})
and on 2-step nilpotent groups (Corollary~\ref{coro:2step}).
In~\S6, we study RB-operators on direct products of groups.

In~\S7, we study extensions of a given map~$\beta$ on a group~$G$
to an RB-operator~$B$ on~$G$.
Applying the necessary condition, we construct and study the auxiliary group~$\bar{G}_{\beta}$.
In~\S8, the problems concerning extensions of Rota---Baxter groups are posted.

In~\S9, we find a~connection between RB-operators on a group and RB-operators on its associated Lie ring.

In~\S10, we state that
on a simple group~$G$ without fixed-point-free automorphisms
every Rota---Baxter operator with trivial kernel
has the form $B(g) = g^{-1}$ (Theorem~\ref{thm:invertibleRB}).
Applying this result and the list of all factorizations of sporadic groups~\cite{Giudici},
we describe all RB-operators on all sporadic simple groups in Theorem~\ref{thm:sporadic}.

\section{Preliminaries}

Let $A$ be an algebra over a field~$\Bbbk$. A linear operator $R$ on $A$ is called
a Rota---Baxter operator of weight~$\lambda\in \Bbbk$ if
\begin{equation}\label{RBAlgebra}
R(x)R(y) = R( R(x)y + xR(y) + \lambda xy ).
\end{equation}
for all $x,y\in A$.
An algebra endowed with a Rota---Baxter operator is called a~Rota---Baxter algebra.

Let us consider an analogue of Rota---Baxter operator of weights~$\pm1$ on a group.

\newpage
\begin{definition}[\cite{Guo2020}]
Let $G$ be a group.

a) A map $B\colon G\to G$ is called a {\it Rota---Baxter operator} of weight~1 if
\begin{equation}\label{RB}
B(g)B(h) = B( g B(g) h B(g)^{-1} )
\end{equation}
for all $g,h\in G$.

b)  A map $C\colon G\to G$ is called a {\it Rota---Baxter operator} of weight~$-1$ if
$$
C(g) C(h) = C( C(g) h C(g)^{-1} g )
$$
for all $g,h\in G$.
\end{definition}

A group endowed with a Rota---Baxter operator is called a~{\it Rota---Baxter group} (RB-group).

\begin{lemma}
If $C\colon G \to G$ is a Rota---Baxter operator of weight $-1$ on a group $G$,
then the operator $\widetilde{C}(g) = g C(g^{-1})$, $g \in G$,
is also a Rota---Baxter operator of weight $-1$.
\end{lemma}

In \cite{Guo2020} it was proved that if $B$ is a Rota---Baxter operator of weight~1
on a group~$G$, then the map $C(g) = B(g^{-1})$ is an Rota---Baxter operator of weight $-1$.
Another connection between Rota---Baxter operators of weights~1 and~$-1$ gives

\begin{proposition}
If $B\colon G \to G$ is Rota---Baxter operator of weight~1 on a group $G$,
then the operator $C(g) = g B(g)$, $g \in G$, is a Rota---Baxter operator of weight $-1$ on $G$.
\end{proposition}

\begin{proof}
One needs to check~\eqref{RB}. By the definition of $C$ it is equivalent to
$$
g B(g) h B(h) = C(g B(g) h B(g)^{-1}),
$$
i.\,e.,
$$
B(g) B(h) = B( g B(g) h B(g)^{-1} ). \qedhere
$$
\end{proof}

Hence, we have a bijection between Rota---Baxter operators of weight 1 and $-1$
on a~group $G$. To the end of the paper we will consider only Rota---Baxter operators
of weight 1 and will call them simply Rota---Baxter operators (RB-operators).

\begin{example}
Let $G$ be a group. Then

a) the map $B_0(g) = e$ is an RB-operator on $G$,

b) the map $B_{-1}(g) = g^{-1}$ is an RB-operator on $G$.
\end{example}

We will call the Rota---Baxter operators $B_0$ and $B_{-1}$ as {\it elementary RB-operators}.
A~group $G$ is called {\it RB-elementary}, if any RB-operator on $G$ is elementary.
It is evident that every cyclic group of prime order is elementary.
Other examples of elementary groups will be given in the last section.

Given a group $G$, we use the usual notation of the commutator $[g,h] = g^{-1}h^{-1}gh$
for $g,h\in G$; the commutator subgroup $[H,L]$ for subgroups $H,L$ of~$G$
is a~subgroup that is generated by the set of commutators $\{[h,l]\mid h\in H,\,l\in L\}$.

Let us state some elementary facts about RB-operators.

\begin{lemma}\label{lem:elementary}
If $B\colon G \to G$ is a Rota---Baxter operator on a group $G$, then

a) $B(e) = e$;

b) $B(g)B(g^{-1}) = B([g^{-1},B(g)^{-1}])$;

c) $B(g)B(B(g)) =  B(g B(g))$;

d) $B(g)^{-1} = B(B(g)^{-1}g^{-1}B(g))$ for any $g \in G$.
\end{lemma}

\begin{proof}
a) It follows from~\eqref{RB} considered with $g = h = e$.

b) It follows from~\eqref{RB} considered with $h = g^{-1}$.

c) It follows from~\eqref{RB} considered with $h = B(g)$.

d) It follows from the equalities
\begin{equation}\label{R(a)^-1}
B(g)B(B(g)^{-1}g^{-1}B(g))
 {=} B(g\underline{B(g)B(g)^{-1}}g^{-1}\underline{B(g)B(g)^{-1}})
 {=} B(gg^{-1})
 = B(e)
 = e.
\end{equation}
Lemma is proved.
\end{proof}

In~\cite{Guo2020} it was noted (without proof) that
given a group $G$ and a Rota---Baxter operator~$B$ on $G$
the sets $\ker B$ and $\Imm B$ are subgroups of $G$.
Let us write down the proof for convenience.

\begin{lemma}[\cite{Guo2020}]\label{KerImSub}
Let $B$ be a Rota---Baxter operator on a group~$G$. Then

a) $\ker B$ is a subgroup of $G$,

b) $\Imm B$ is a subgroup of $G$.
\end{lemma}

\begin{proof}
a) We have $B(e) = e$ by~Lemma~\ref{lem:elementary}a.
Let $g,h\in \ker B$, then by~\eqref{RB}
\begin{gather*}
B(g^{-1}) = B(g)B(g^{-1}) = B(gB(g)g^{-1}B(g)^{-1}) = B(e) = e, \\
e = B(g)B(h) = B(gB(g)hB(g)^{-1}) = B(gh).
\end{gather*}

b) Directly by~\eqref{RB} we have that if $g,h\in \Imm B$, then $gh\in \Imm B$.
Suppose that $g\in \Imm B$, then $g^{-1}\in \Imm B$ by~Lemma~\ref{lem:elementary}d.
\end{proof}

When a Rota---Baxter operator $B$ on a group $G$ is invertible, i.\,e.,
the map $B^{-1}\colon G\to G$ exists, then $B^{-1}$
is a~crossed homomorphism from $G$ to itself~\cite{Guo2020}.
The following result can be useful in attempt to construct
RB-operators $B$ on a given group with $\ker B\neq\{e\}$.

\begin{lemma} \label{kernelCosets}
Let $B$ be a Rota---Baxter operator on a group~$G$.
If $B(g) = e$ for some $g\in G$, then $B(h) = B(gh)$ for any $h\in G$. In particular, if
$$
G = \coprod_{i \in I} \ker(B)g_i
$$
is the decomposition of $G$ in the disjoint union of cosets,
then $B(x) = B(y)$ if $x$ and $y$ lie in the same coset.
\end{lemma}

\begin{proof}
It follows from~\eqref{RB} considered with such $g$ and $h$.
\end{proof}

The following claim gives a possibility to construct new RB-operators from the known ones.

\begin{lemma}[\cite{Guo2020}]\label{lem:tilde}
Let $B$ be a Rota---Baxter operator on a group~$G$.
Then the operator $\widetilde{B}(g) = g^{-1}B(g^{-1})$
is also a~Rota---Baxter operator on a group~$G$.
\end{lemma}

In particular, $\widetilde{B_0} = B_{-1}$.
It is easy to see that $\widetilde{\!\widetilde{B}} = B$.

The following observation is very important when
one is interested on the classification of
all RB-operators on a given group.

\begin{lemma}\label{lem:Aut}
Let $B$ be a Rota---Baxter operator on a group~$G$.
Let $\varphi$ be an automorphism of $G$.
Then $B^{(\varphi)} = \varphi^{-1}B\varphi$
is a~Rota---Baxter operator on a~group~$G$.
\end{lemma}

\begin{proof}
The following equalities imply the statement,
\begin{multline*}
\varphi( B^{(\varphi)}(g) B^{(\varphi)}(h) )
 = \varphi( \varphi^{-1}(B(\varphi(g)))\varphi^{-1}(B(\varphi(h))) )
 = B(\varphi(g))B(\varphi(h)) \\
 = B(\varphi(g)B(\varphi(g))\varphi(h)B(\varphi(g))^{-1} )
 = \varphi\varphi^{-1}B\varphi( g\varphi^{-1}B(\varphi(g))h\varphi^{-1}(B(\varphi(g))^{-1}) ) \\
 = \varphi( B^{(\varphi)}(gB^{(\varphi)}(g)h(B^{(\varphi)}(g))^{-1}) ). \qedhere
\end{multline*}
\end{proof}

Thus, we may describe Rota---Baxter operators
on a~group up to conjugation with its automorphism.

\begin{lemma}\label{tildeB-properties}
Let $G$ be a group, let $\varphi$ be an automorphism $\varphi$ of $G$
and let $B$ be a~Rota---Baxter operator on $G$. Then
$(\widetilde{B})^{(\varphi)} = \widetilde{B^{(\varphi)}}$.
\end{lemma}

\begin{proof}
We have
$$
(\widetilde{B})^{(\varphi)} (g)
 = \varphi^{-1}( \widetilde{B} (\varphi(g))) )
 = \varphi^{-1}( (\varphi(g))^{-1}B( (\varphi(g))^{-1} ) )
 = g^{-1}B^{(\varphi)}(g^{-1})
 = \widetilde{B^{(\varphi)}}(g). \qedhere
$$
\end{proof}

Let $(G, B)$ and $(G', B')$ be Rota---Baxter groups. A map $\Phi \colon G \to G'$
is a Rota---Baxter group homomorphism if $\Phi$ is a group homomorphism such that
$\Phi B = B'\Phi$.

A group $G$ is called factorizable if $G = HL$ for some its subgroups $H$ and $L$.
The expression $G = HL$ is called a factorization of $G$.
If additionally $H\cap L = \{e\}$, then such factorization is called exact.
If $H$ and $L$ are proper subgroups of~$G$, we call a factorization $G = HL$ as a proper one.

\section{Rota---Baxter induced group}

In \cite{Guo2020} a new binary operation $\circ \colon G \to G$
on a~Rota---Baxter group $(G,B)$ was defined.

\begin{proposition}[\cite{Guo2020}]\label{prop:Derived}
Let $(G,\cdot,B)$ be a Rota---Baxter group.

a) The pair $(G, \circ )$ with the product
\begin{equation}\label{R-product}
g\circ h = gB(g)hB(g)^{-1},
\end{equation}
where $g,h\in G$, is also a group.

b) The operator $B$ is a Rota---Baxter operator on the group $(G,\circ)$.

c) The map $B\colon (G,\circ) \to (G,\cdot)$
is a homomorphism of Rota---Baxter groups.
\end{proposition}

We will  denote the group $(G,\circ)$ as $G_B$.

Given a group~$G$ and a Rota---Baxter operator~$B$ on~$G$,
consider a map $B_+$ on $G$ defined as follows, $B_+(g) = gB(g)$~\cite{Guo2020}.
Note that $B_+$ is a group homomorphism from $G_B$ to $G$.
Lemma~\ref{lem:elementary}c says that $B_+B = BB_+$.

\begin{example}
Let $G$ be a group with a Rota---Baxter operator $B$.

a) If $B = B_0$, then
$$
g\circ h = gB_0(g)hB_0(g)^{-1} = g h.
$$
Hence, both products $\circ$ and $\cdot$ coincide.

b) If $B = B_{-1}$, then
$$
g\circ h = gB_{-1}(g) h B_{-1}(g)^{-1} = h g.
$$
Hence, $(G, \circ)$ is a group with the opposite product.

c) If $G$ is abelian, then $g\circ h = g\cdot h$ for all $g,h\in G$.
\end{example}

In~\cite{Guo2020}, L. Guo, H. Lang, and Yu. Sheng proved the results analogous
to the ones stated by M.A.~Semenov-Tyan-Shanskii for Lie algebras~\cite{Semenov83}.

\begin{proposition}[\cite{Guo2020}] If $(G, B)$ is an RB-group, then

a) $\ker(B)$ and $\ker(B_+)$ are normal in $G_B$;

b) $\ker(B) \unlhd \Imm(B_+)$ and $\ker(B_+) \unlhd \Imm(B)$ in $G$;

c) We have the isomorphism of quotient groups
\begin{equation}\label{FactorIso}
\Imm(B_+)/\ker(B)\cong \Imm(B)/\ker(B_+);
\end{equation}

d) We have the factorization
\begin{equation}\label{ImageFactorization}
G = \Imm(B_+)\Imm(B).
\end{equation}
\end{proposition}

Let us find the exact formulas for words written in terms of the operation $\circ$ in $G_B$.
Given an integer~$k$, denote by $a^{\circ(k)}$ the $k$th power of $a$ under the product $\circ$.

\begin{proposition}\label{for}
Let $(G, B)$ be a Rota---Baxter group.
If $A$ is a subset of $G$ and
$$
w = a_{i_1}^{\circ(k_1)} \circ a_{i_2}^{\circ(k_2)} \circ \ldots
    \circ a_{i_s}^{\circ(k_s)},\quad a_{i_j} \in A,\ k_j \in \mathbb{Z},
$$
is presented by a word under the operation $\circ$, then
\begin{equation}\label{WordInG_B}
w {=} (a_{i_1}B(a_{i_1}))^{k_1}
    (a_{i_2}B(a_{i_2}))^{k_2} \ldots
    (a_{i_s}B(a_{i_s}))^{k_s}
    B(a_{i_s})^{-k_s}
    B(a_{i_{s-1}})^{-k_{s-1}} \ldots
    B(a_{i_1})^{-k_1}.
\end{equation}
In particular, $a^{\circ(k)} = (B_+(a))^k(B(a))^{-k}$ for every integer $k$ and
$a^{\circ(-1)} = B(a)^{-1} a^{-1} B(a)$.
\end{proposition}

\begin{proof}
Firstly, it is easy to show by induction on $n\in\mathbb{N}$ that
$a^{\circ(n)} = (B_+(a))^n(B(a))^{-n}$. Indeed, for $n = 0$ we have $e = e$.
The induction step follows from the equalities
$$
a^{\circ(n)}
 = a\circ (a^{\circ(n-1)})
 = aB(a)a^{\circ(n-1)}B(a)^{-1}
 = (aB(a))^n B(a)^{-n}.
$$
One can check directly that $B(a)^{-1}a^{-1}B(a)$
is the inverse element to~$a$ in the group $G_B$.
Now, again by induction on $n\in\mathbb{N}$ and
by Lemma~\ref{lem:elementary}d we compute
\begin{multline*}
a^{\circ(-n)}
 = a^{\circ(-1)}\circ(a^{\circ(-(n-1))}) \\
 = B(a)^{-1}a^{-1}\underline{B(a)B( B(a)^{-1}a^{-1}B(a) )}
 (aB(a))^{-(n-1)} B(a)^{n-1} (B( B(a)^{-1}a^{-1}B(a) ) )^{-1} \\
 = (aB(a))^{-1}(aB(a))^{-(n-1)} B(a)^{n-1}B(a)^1
 = (aB(a))^{-n} B(a)^{n}.
\end{multline*}

Denote $w' = a_{i_2}^{\circ(k_2)} \circ \ldots \circ a_{i_s}^{\circ(k_s)}$.
Applying that $B$ is a homomorphism from $G_B$ to $G$, we get
$$
w = a_{i_1}^{\circ(k_1)}\circ w'
  = (B_+(a_{i_1}))^{k_1}\underline{(B(a_{i_1}))^{-k_1}
  B(a_{i_1}^{\circ(k_1)})}w'B(a_{i_1}^{\circ(k_1)})^{-1}
  = (B_+(a_{i_1}))^{k_1}w'(B(a_{i_1}))^{-k_1}.
$$
Thus, by induction on~$s$ we derive the formula~\eqref{WordInG_B}.
\end{proof}

\section{Constructions via factorizations}

The next example allows us to construct non-elementary RB-operators.

\begin{example}[\cite{Guo2020}]\label{exm:split}
Let $G$ be a group. Given an exact factorization $G = HL$,
a map $B\colon G\to G$ defined as follows,
$$
B(hl) = l^{-1}
$$
is a Rota---Baxter operator on~$G$.
\end{example}

Let us call such Rota---Baxter operator on~$G$ as a~{\it splitting Rota---Baxter operator}.
Note that in this case $G_B\cong H\times L$ and $B_+(hl) = h$.

\begin{proposition}\label{SplittingCond}
Let $G$ be a group and let $B\colon G\to G$ be an RB-operator on~$G$.
Then $B$ is a splitting Rota---Baxter operator on~$G$
if and only if $B(gB(g)) = e$ for all $g\in G$.
\end{proposition}

\begin{proof}
Suppose that $B$ is a splitting RB-operator on~$G$ with subgroups~$H$~and~$L$.
Take $g = hl$, where $h\in H$ and $l\in L$. Then
$$
B(gB(g)) = B(hlB(hl)) = B(hll^{-1}) = B(h) = e.
$$

Conversely, suppose that $B(gB(g)) = e$ for all $g\in G$.
Let us prove that $G = \ker(B)\Imm(B)$ and this factorization is exact.
Firstly, let $a\in \ker(B)\cap\Imm(B)$.
Then $B(a) = e$ and there exists $g\in G$ such that $a = B(g)$.
By the assumption, we have $e = B(gB(g)) = B(ga)$.
Since $ga,a\in \ker(B)$, by Lemma~\ref{KerImSub}a we conclude that $g\in \ker(B)$.
So, $a = e$.

Secondly, let $g\in G$. The equality $B(gB(g)) = e$ implies that $gB(g) = a \in\ker(B)$.
Thus, $g = ab$, where $b = B(g)^{-1}\in\Imm(B)$ by Lemma~\ref{KerImSub}b.
\end{proof}

\begin{remark}
By Lemma~\ref{lem:elementary}, the condition~$B(gB(g)) = e$
is equivalent to the relation
$B(g) B^2 (g) = e$ or $B^2 (g) = B(g)^{-1}$.
It means that if $B$ is a splitting RB-operator,
then $B$ inverts all elements of $\Imm(B)$.
\end{remark}

Let us extend Example~\ref{exm:split} on some triple factorizations.

\begin{proposition}\label{triangular}
Let $G$ be a group such that $G = HLM$, where $H$, $L$, and $M$
are subgroups of $G$ with pairwise trivial intersection.
Let $C$ be a Rota---Baxter operator on~$L$.
Moreover, $[H,L] = [C(L),M] = e$.
Then the map $B\colon G\to G$ defined by the formula
$$
B(hlm) = C(l)m^{-1}
$$
is a Rota---Baxter operator on~$G$.
\end{proposition}

\begin{proof}
Let $h,h'\in H$, $l,l'\in L$, and $m,m'\in M$. Then
\begin{multline*}
C(l)C(l')m^{-1}m'^{-1}
 = C(l)m^{-1}C(l')m'^{-1}
 = B(hlm)B(h'l'm') \\
 = B(hlm B(hlm)h'l'm'B(hlm)^{-1})
 = B(hlm C(l)m^{-1}h'l'm'mC(l)^{-1}) \\
 = B(hl C(l)\underline{mm^{-1}}h'l'C(l)^{-1}m'm)
 = B(hh'l C(l)l'C(l)^{-1}m'm) \\
 = C(lC(l)l'C(l)^{-1})m^{-1}m'^{-1},
\end{multline*}
which is fulfilled by~\eqref{RB}.
\end{proof}

Let us call a Rota---Baxter operator on $G = HLM$
defined by Proposition~\ref{triangular} with the help of $C$
as a {\it triangular-splitting} one.
Note that $G_B\cong H\times L_C\times M^{\op}$, where
$M^{\op}$ is the set $M$ with the opposite product $m*m' = m'm$.

\begin{corollary}
Let $G = \langle a,b \rangle$ be two generated 2-step nilpotent group. Put $c = [b,a]$.
Then $G = \langle a \rangle \langle c \rangle \langle b \rangle$ and for any integer~$k$ the map
$$
B(a^{\alpha} c^{\beta} b^{\gamma})
 = c^{k\beta} b^{-\gamma},\quad \alpha,\beta,\gamma \in \mathbb{Z},
$$
is an RB-operator on~$G$.
\end{corollary}

Given a~splitting RB-operator $B$ on a~group~$G$, we have an exact factorization
\begin{equation}\label{KerImFactor}
G = \ker(B)\Imm(B).
\end{equation}
It is easy to see that if $G$~is abelian, then
there exists a non-splitting RB-operator~$B$ on~$G$ such that~\eqref{KerImFactor} is fulfilled.
Indeed, if $G$~is a~direct sum of abelian subgroups: $G = H \times L$, then any map
$hl\mapsto \varphi(l)$, where $\varphi$~is an endomorphism of~$L$,
defines an RB-operator on~$G$ satisfying~\eqref{KerImFactor}.

The next claim generalizes this observation.

\begin{proposition}
Given a semidirect product $G = H\rtimes L$, let $C$ be a~Rota---Baxter operator on $L$.
Then a map $B\colon G\to G$ defined by the formula $B(hl) = C(l)$, where $h\in H$ and $l\in L$,
is a~Rota---Baxter operator.
\end{proposition}

\begin{proof}
We state the claim by the following computations for $h,h'\in H$ and $l,l'\in L$,
$$
B(hl)B(h'l') = C(l)C(l'),
$$

\vspace{-0.9cm}
\begin{multline*}
B(hlB(hl)h'l'B(hl)^{-1})
 = B(hl C(l)h'l'C(l)^{-1})
 = B(hh'^{C(l)^{-1}l^{-1}}l C(l)l'C(l)^{-1}) \\
 = C(lC(l)l'C(l)^{-1})
 = C(l)C(l'). \qedhere
\end{multline*}
\end{proof}

Note that $G_B\cong H\rtimes L_C$.

\section{Constructions via homomorphisms, projections, and exact formulas}

When $G$ is an abelian group, then $B$ is a Rota---Baxter operator on~$G$
if and only if $B$~is an endomorphism of $G$.

\begin{proposition}\label{prop:RB-Hom}
a) Let $(G, B)$ be a Rota---Baxter group and $B$ be an automorphism of $G$. Then $G$ is abelian.

b) If $G$ a group and $H$ is its abelian subgroup,
then any homomorphism (or antihomomorphism) $B\colon G \to H$ is a Rota---Baxter operator.
\end{proposition}

\begin{proof}
Suppose that a Rota---Baxter operator $B\colon G \to G$ is an endomorphism, then
$$
\underline{B(g)}B(h)
 = B(gB(g)hB(g)^{-1})
 = \underline{B(g)}B^2(g) B(h) (B^2(g))^{-1},\quad g,h \in G.
$$
It is equivalent to the relation
\begin{equation}\label{RB-Hom}
[B^2(g), B(h)] = e.
\end{equation}
Hence, b) is true.
The proof for the case when $B$ is an antihomomorphism from $G$ to $H$ is analogous.

If further $B$ is an automorphism, then $G$~is an abelian group by~\eqref{RB-Hom}.
\end{proof}

The following several constructions of Rota---Baxter operators work only for groups very close to abelian ones.

Given an integer~$k$, a group $G$ is called $k$-abelian~\cite{Levi} if $(gh)^k = g^k h^k$ for all $g,h\in G$.

\begin{proposition}\label{prop:n-abelian}
Let $G$ be a group, and fix a natural~$n$. A map $B\colon G\to G$ defined as follows
$B(g) = g^n$ is a Rota---Baxter operator on $G$ if and only if $G$ is $(n+1)$-abelian.
\end{proposition}

\begin{proof}
Given $g,h\in G$, the identity~\eqref{RB} for the map $B(x) = x^n$ has the form
$$
g^n h^n
 = (g^{n+1}hg^{-n})^n
 = g^{n+1}hg^{-n}g^{n+1}hg^{-n}g^{n+1}\ldots hg^{-n}
 = g^n(gh)^n g^{-n}.
$$
So, it is equivalent to the equality
$(gh)^n = h^n g^n$ or to the following one:
$$
(hg)^{n+1} = h(gh)^n g = h^{n+1}g^{n+1}.
$$
It means that $B$ is a Rota---Baxter operator on $G$ if and only if $G$ is $(n+1)$-abelian.
\end{proof}

Since every $k$-abelian group is also $(1-k)$-abelian group,
we may represent the induced product in $G_B$ as follows
$$
g\circ h
 = gg^nhg^{-n}
 = g^{n+1}h^{n+1}h^{-n}g^{-n}
 = (gh)^{n+1}(hg)^{-n}.
$$

\begin{remark}
The statement of Proposition~\ref{prop:n-abelian} in ``if'' direction
follows from Proposition~\ref{prop:RB-Hom},
since the $n$th power of a~$(n+1)$-abelian group~$G$ lies in $Z(G)$~\cite{Kaluzhnin}.
\end{remark}

\begin{proposition}\label{prop:center-conj}
Let $G$ be a group.
Then a map $B(x) = g^{-1} x^{-1} g$ is a Rota---Baxter operator on $G$
if and only if $[g,G]\subset Z(G)$.
\end{proposition}

\begin{proof}
Suppose that $B(x) = g^{-1} x^{-1} g$ is a Rota---Baxter operator on $G$. It means that
$$
g^{-1} x^{-1}y^{-1} g
 = g^{-1}( x g^{-1}x^{-1}g y g^{-1} x g )^{-1}g
 = g^{-1}g^{-1}x^{-1}gy^{-1}g^{-1}xg x^{-1}g
$$
for all $x,y\in G$. This equality is equivalent to the following one,
$$
x^{-1}y^{-1}
 = [g,x]x^{-1}y^{-1}[g,x^{-1}].
$$
Denote $s = x^{-1}y^{-1}$, then we have
$[x,g]^s = [g,x^{-1}]$ for all $x,s\in G$.
Fixing $x\in G$, we conclude that $[g,x]\in Z(G)$ for every $x\in G$.

Conversely, suppose that $g\in G$ such an element that $[g,x]\in Z(G)$ for every $x\in G$.
By the previous part, it is enough to state that $[x,g] = [g,x^{-1}]$ for all $x,s\in G$.
Since
$$
[g,x^{-1}]
 = g^{-1}xgx^{-1}
 = x^{-1}g^{-1}xg[g^{-1}xg,x^{-1}]
 = [x,g]\cdot [x^g,x^{-1}],
$$
we have to prove that $[x^g,x^{-1}] = e$. Denote $c = [g,x^{-1}]$.
We get it by the formulas
$$
[x^g,x^{-1}]
 = [[g,x^{-1}]x,x^{-1}]
 = [cx,x^{-1}]
 = x^{-1}c^{-1}xcxx^{-1}
 = e. \qedhere
$$
\end{proof}

Given a group~$G$ with an element $g$ such that $[g,G]\subset Z(G)$,
we may compute the product of $G_B$, where $B(x) = g^{-1} x^{-1} g$:
$$
x\circ y
 = xB(x)yB(x)^{-1}
 = xg^{-1} x^{-1} gyg^{-1} xg
 = yxg^{-1} x^{-1} gg^{-1} xg
 = yx,
$$
so this product coincides with the product induced by the RB-operator $B_{-1}$.

\begin{corollary}\label{coro:2step}
Let $G$ be a 2-step nilpotent group, $g \in G$,
then the map $B_g(x) = g^{-1} x^{-1} g$ is a Rota---Baxter operator.
\end{corollary}

It is evident that this map $B_g$ is a bijection.
Denote by $\mathcal{RB}(G)$ the set of all Rota---Baxter operators on~$G$.
Hence, we have a map
$\varphi \colon G/[G,G]\to \mathcal{RB}(G)$ defined as follows, $\varphi(g) = B_g$.

Let us try to construct a Rota---Baxter operator by the formula $B(g) = agb$
with the help of some fixed $a$ and $b$. In comparison to Proposition~\ref{prop:center-conj}
we allow $a$ and $b$ be not necessary inverse to each other, however we do not inverse $g$.

\begin{corollary}
Let $G$ be a group, and fix $a,b\in G$. A map $B\colon G\to G$ defined as follows
$B(g) = agb$, is a Rota---Baxter operator on $G$
if and only if $G$ is abelian and $b = a^{-1}$, i.\,e., $B = \id$.
\end{corollary}

\begin{proof}
If $G$ is abelian, then every endomorphism of $G$ is
a Rota---Baxter operator on $G$, including the identity map.

Conversely, suppose that a map $B$ on $G$ defined by the formula
$B(g) = agb$ is a~Rota---Baxter operator on $G$.
Since $B(e) = e$, we have $b = a^{-1}$. Thus, $B$ is an automorphism of~$G$
and by~Proposition~\ref{prop:RB-Hom}a we have that $G$ is abelian, and so $B = \id$.
\end{proof}

\begin{definition}
A Rota---Baxter group $(G, B)$ has the {\it deep} $m$
($m$ is a natural number or $\omega$) if it is a minimal such that
$$
G > B(G) > B^2(G) > \cdots > B^m(G) = B^{m+1}(G).
$$
\end{definition}

Note that the RB-group $(G,B_{-1})$ has deep 0
and the RB-group $(G,B)$, where $|G|>1$, $B = B_0$ or $B$ is splitting, has deep 1.

The following example shows that there are RB-groups of infinite deep.

\begin{example}
Let $(\mathbb{Z}, +)$ be the infinite cycle group and
$B\colon \mathbb{Z} \to \mathbb{Z}$, $B(x) = 2 x$ for any $x \in \mathbb{Z}$.
Then $(\mathbb{Z}, B)$ is a Rota---Baxter group and
it has the deep $\omega$ since
$$
\mathbb{Z}
 > B(\mathbb{Z}) = 2\mathbb{Z}
 > B^2(\mathbb{Z}) = 4 \mathbb{Z}
 > \dots
 > B^{\omega}(\mathbb{Z}) = \bigcap_{k=1}^{\infty}B^{k}(\mathbb{Z}) = 0.
$$
\end{example}

If $F_n = \langle x_1, x_2, \ldots, x_n  \rangle$ is a free non-abelian group,
then we can define its projection $\pi_1 \colon F_n \to \langle x_1 \rangle$ by
$$
\pi_1(x_1) = x_1,\quad \pi_1(x_i) = e,\quad i\neq1.
$$
Then $F_n = \ker(\pi_1) \langle x_1 \rangle$ and for any integer $k$
the map $B(g) = \pi_1(g)^k$ is a Rota---Baxter operator on $F_n$.
The following proposition generalizes this observation.

\begin{proposition}\label{fr}
Let $G = K * L$ be a free product with abelian $L$ and $\pi_1\colon G \to L$
be a projection on the second component. Then for any endomorphism $\chi$ of $L$ the map
$B(g) = \chi(\pi_1(g))$ is a~Rota---Baxter operator on $G$.
\end{proposition}

\begin{question}
Let $G$ be a group and $H$ be its a proper subgroup. Under which conditions
there is a Rota---Baxter operator $B$ such that $B(G) = H$?
In particular, if $F$ is a free non-abelian group, is there an RB-operator
on~$F$ such that its image is the commutator subgroup $F'$ of $F$?
\end{question}

\section{Constructions via direct products}

Using RB-operators on some groups, we can define RB-operators
on direct products of these groups, it follows from

\begin{proposition} \label{direct}
Let $G$ be a group such that $G = H\times L$.

a) Let $B_H$ be a Rota---Baxter operator on $H$
and let $B_L$ be a Rota---Baxter operator on $L$.
Then a map $B\colon G\to G$ defined by the formula
$B(hl) = B_H(h)B_L(l)$, where $h\in H$ and $l\in L$,
is a Rota---Baxter operator.

b) If $L$ is abelian and $|L|>2$, then
there exists a non-splitting Rota---Baxter operator on $G$.
\end{proposition}

\begin{proof}
a) Straightforward.

b) On the contrary, suppose that every Rota---Baxter operator on $G$ is splitting.
Define a Rota---Baxter operator~$B(\psi)$ on~$G$
with the help of an automorphism~$\psi$ of $L$ as follows,
$B(\psi)((h,l)) = \psi(l)$, by Proposition~\ref{prop:RB-Hom}
it is a~Rota---Baxter operator on~$G$.
By Proposition~\ref{SplittingCond}, we should have
$B(lB(l)) = \psi(l)\psi^2(l) = e$ for all $l\in L$.
Take $\psi = \id$, so we have $l^2 = e$ for all $l\in L$.
Since $L$ is abelian, we get that $L$ is a~direct sum of copies of $\mathbb{Z}_2$.
By the condition $|L|>2$, so we may find nonzero and different $l_1,l_2\in L$.
Let us consider an automorphism~$\psi$ of~$L$ which interchanges $l_1$ and $l_2$
and does not change all other elements.
It is easy to see that $B(\psi)$ is non-splitting. We arrive at a~contradiction.
\end{proof}

It is easy to check that the following example gives Rota---Baxter operators on $G^n$,
below we will generalize it.

\begin{example}\label{exm:direct}
Let $G^n = G\times G\times \ldots \times G$.
Then the following maps from $G^n$ to $G^n$
\begin{gather*}
B((g_1,\ldots,g_n)) = (e,g_1,g_2g_1,g_3g_2g_1,\ldots,g_{n-1}g_{n-2}\ldots g_1), \\
\widetilde{B}((g_1,\ldots,g_n)) = (g_1^{-1},g_2^{-1}g_1^{-1},g_3^{-1}g_2^{-1}g_1^{-1},\ldots,
 g_n^{-1}g_{n-1}^{-1}g_{n-2}^{-1}\ldots g_1^{-1})
\end{gather*}
are Rota---Baxter operators.
\end{example}

These Rota---Baxter operators may be extended to the ones
on the infinite Cartesian product of $G$.

Let $R$ be a~Rota---Baxter operator of weight~1 on
the direct sum of fields (considered as associative algebra)
$\Bbbk^n = \Bbbk e_1\oplus \Bbbk e_2\oplus \ldots \oplus \Bbbk e_n$, where $e_ie_j = \delta_{ij}e_i$.
A~linear operator $R(e_i) = \sum\limits_{k=1}^n r_{ik}e_k$,
$r_{ik}\in \Bbbk$, is an RB-operator of weight~1 on~$\Bbbk^n$
if and only if the following conditions are satisfied~\cite{AnBai,Braga}:

(1) $r_{ii} = 0$ and $r_{ik}\in\{0,1\}$
or $r_{ii} = -1$ and $r_{ik}\in\{0,-1\}$ for all $k\neq i$;

(2) if $r_{ik} = r_{ki} = 0$ for $i\neq k$,
then $r_{il}r_{kl} = 0$ for all $l\not\in\{i,k\}$;

(3) if $r_{ik}\neq0$ for $i\neq k$,
then $r_{ki} = 0$ and
$r_{kl} = 0$ or $r_{il} = r_{ik}$ for all $l\not\in\{i,k\}$.

Moreover, we may suppose that the matrix of $R$ is upper-triangular~\cite{Braga,SumOfFields}.

\begin{theorem}\label{theo:directProduct}
Let $G^n = G\times G\times \ldots \times G$ and let $\Bbbk$ be a~field.
Let $R$ be a~Rota---Baxter operator of weight~1 on
$\Bbbk e_1\oplus \Bbbk e_2\oplus \ldots \oplus \Bbbk e_n$,
$R(e_i) = \sum\limits_{i=1}^n r_{ik}e_k$, and the matrix of $R$ is upper-triangular.
Define a~map $B\colon G^n\to G^n$ as follows,
\begin{equation} \label{RBOnDirectProduct}
B((g_1,\ldots,g_n)) = (t_1,\ldots,t_n),\quad
t_i = g_i^{r_{ii}}g_{i-1}^{r_{i-1\,i}}\ldots g_1^{r_{1i}}.
\end{equation}
Then $B$ is a Rota---Baxter operator on $G^n$.
\end{theorem}

\begin{proof}
Write down the left and right hand sides of~\eqref{RB} for~$B$, where
$g = (g_1,\ldots,g_n)$ and $h = (h_1,\ldots,h_n)$:
\begin{gather}
B(g)B(h)
 = \big(g_1^{r_{11}}h_1^{r_{11}},g_2^{r_{22}}g_1^{r_{12}}h_2^{r_{22}}h_1^{r_{12}},\ldots,
 g_n^{r_{nn}}g_{i-1}^{r_{n-1\,n}}\ldots g_1^{r_{1n}}
 h_n^{r_{nn}}h_{i-1}^{r_{n-1\,n}}\ldots h_1^{r_{1n}}\big), \label{LeftRBSumFields} \\
B(gB(g)hB(g)^{-1})
 = \big(t_1^{r_{11}},t_2^{r_{22}}t_1^{r_{12}},\ldots,
 t_n^{r_{nn}}t_{n-1}^{r_{n-1\,n}}\ldots t_1^{r_{1n}}\big), \label{RightRBSumFields}
\end{gather}
where
\begin{gather*}
t_1 = g_1g_1^{r_{11}}h_1g_1^{-r_{11}}, \quad
t_2 = g_2g_2^{r_{22}}g_1^{r_{12}}h_2g_1^{-r_{12}}g_2^{r_{22}}, \quad \ldots, \\
t_n = g_n g_n^{r_{nn}}g_{n-1}^{r_{n-1\,n}}\ldots g_1^{r_{1n}}h_n g_1^{-r_{1n}}\ldots
 g_{n-1}^{-r_{n-1\,n}}g_n^{-r_{nn}}.
\end{gather*}

Let us prove by induction on $i=1,\ldots,n$ that the $i$-coordinates
of~\eqref{LeftRBSumFields} and~\eqref{RightRBSumFields} are equal.
For $i = 1$, we have
$$
g_1^{r_{11}}h_1^{r_{11}}
 = \big(g_1g_1^{r_{11}}h_1g_1^{-r_{11}}\big)^{r_{11}},
$$
since when $r_{11} = 0$ it is trivial equality $e = e$,
and when $r_{11} = -1$ we get $g_1^{-1}h_1^{-1} = (h_1g_1)^{-1}$.

Suppose that we have proved that $i$-coordinates
of~\eqref{LeftRBSumFields} and~\eqref{RightRBSumFields} are equal
for all $i<k$, where $2\leq k\leq n$. Let us prove the equality
$$
g_i^{r_{ii}}g_{i-1}^{r_{i-1\,i}}\ldots g_1^{r_{1i}}
h_i^{r_{ii}}h_{i-1}^{r_{i-1\,i}}\ldots h_1^{r_{1i}}
 = t_i^{r_{ii}}t_{i-1}^{r_{i-1\,i}}\ldots t_1^{r_{1i}}
$$
for $i = k$.
If $r_{ii} = 0$, then we may apply the induction hypothesis for $i-1$ and
for the coefficients $r'_{s\,i-1} = r_{si}$.
If $r_{ii} = -1$, then we again apply the induction hypothesis for $i-1$ and
for the coefficients $r'_{s\,i-1} = r_{si}$ to derive that
\begin{multline*}
t_i^{r_{ii}}t_{i-1}^{r_{i-1\,i}}\ldots t_1^{r_{1i}}
 = (\underline{g_i g_i^{-1}}g_{i-1}^{r_{i-1\,i}}\ldots g_1^{r_{1i}}h_i g_1^{-r_{1i}}
 \ldots g_{i-1}^{-r_{i-1\,i}}g_i)^{-1}
 t_{i-1}^{r_{i-1\,i}}\ldots t_1^{r_{1i}} \\
 = g_i^{-1}g_{i-1}^{r_{i-1\,i}}\ldots g_1^{r_{1i}}h_i^{-1}\underline{g_1^{-r_{1i}}
 \ldots g_{i-1}^{-r_{i-1\,i}}
 g_{i-1}^{r_{i-1\,i}}\ldots g_1^{r_{1i}}}h_{i-1}^{r_{i-1\,i}}\ldots h_1^{r_{1i}} \\
 = g_i^{r_{ii}}g_{i-1}^{r_{i-1\,i}}\ldots g_1^{r_{1i}}
h_i^{r_{ii}}h_{i-1}^{r_{i-1\,i}}\ldots h_1^{r_{1i}}. \qedhere
\end{multline*}
\end{proof}

\begin{remark}
Analogously to~\cite{SumOfFields}, we may state that $G^n_B\cong G^n$.
\end{remark}

\begin{corollary}
We may slightly modify~\eqref{RBOnDirectProduct} as follows.
Let $\psi_2,\ldots,\psi_n\in\Aut(G)$. Then a~map $P\colon G^n\to G^n$
defined by the formula
\begin{gather*}
P((g_1,\ldots,g_n)) = (t_1,\ldots,t_n),\\
t_1 = g_1^{r_{11}},\quad
t_i = g_i^{r_{ii}}\psi_i\big(g_{i-1}^{r_{i-1\,i}}\psi_{i-1}\big(g_{i-2}^{r_{i-2\,i}}
 \ldots \psi_2\big(g_1^{r_{1i}}\big)\big)\big),\ i\geq2,
\end{gather*}
is a Rota---Baxter operator on $G^n$.
\end{corollary}

\begin{proof}
Define an automorphism $\varphi$ of $G^n$ as follows,
$$
\varphi((g_1,g_2,g_3,\ldots,g_n))
 = \big(g_1,\psi_2^{-1}(g_2),\psi_2^{-1}\psi_3^{-1}(g_3),\ldots,\psi_2^{-1}\ldots \psi_n^{-1}(g_n)\big).
$$
By Lemma~\ref{lem:Aut} we have that $B^{(\varphi)}$ is a Rota---Baxter operator on $G^n$,
where $B$ is defined by~\eqref{RBOnDirectProduct}. Note that $P = B^{(\varphi)}$.
\end{proof}

\begin{corollary}
Given a group $G = H\times H\times L$ for some $L$ and a non-trivial group~$H$,
the map $B\colon G\to G$ defined as follows,
$B((h_1,h_2,l)) = (e,h_1,e)$ is a non-splitting Rota---Baxter operator.
\end{corollary}

\begin{proof}
Applying Example~\ref{exm:direct} and Proposition~\ref{prop:RB-Hom},
we conclude that $B$ is a Rota---Baxter operator on $G$.
Denote $g = (h_1,h_2,l)$. Since $B(gB(g)) = (e,h_1,e_L)$ is not the identity of~$G$
when $h_1\neq e$, the operator $B$ is not splitting by~Proposition~\ref{SplittingCond}.
\end{proof}

\section{Extensions of maps to RB-operators}

In the previous sections we have considered some possibilities
to construct RB-operators on  groups.
In general case we formulate

\begin{question}
Let $G = \langle A \rangle$ be a group generated by a set $A = \{ a_i \mid i \in I \}$.
Suppose that  $U = \{u_i\}_{i \in I}$ is a subset of $G$ with the same cardinality as $A$.
Under which conditions the map $\beta \colon A \to U$, $a_i \mapsto u_i$, $i \in I$,
can be (uniquely) extended to an RB-operator $B \colon G \to G$?
\end{question}

Suppose at first that we know that $(G, B)$ is an RB-group.
What is a minimal subset $M \subseteq G$ such that the values $B(m)$
for all $m \in M$ uniquely define the values $B(g)$ for all elements of $G$?

The next theorem gives the answer on this question.

\begin{theorem} \label{t}
Let $G$ be a group, let $A = \{a_i\}_{i \in I}$ be some its subset,
and let $B$ an RB-operator on $G$.

a) We can find the values of $B$ on all elements of the subgroup
$\langle A \rangle_{B} \leq (G, \circ)$ by the values $B(a_i)$, $a_i \in A$.
The image $B(\langle A \rangle_{B})$ is generated by elements $B(a_i)$, $i \in G$.

b) If $A$ is a generating set of $G_B$,
then the values $B$ on $A$ uniquely define the values of~$B$ on all elements of $G$.
\end{theorem}

\begin{proof}
a) By assumption $(G, B)$ is a Rota---Baxter group and $B(a_i) = u_i$, $i\in I$.
Suppose that an element $g$ of $\langle A \rangle_{B}$ is presented by some word
$$
w = a_{i_1}^{\circ(k_1)} \circ a_{i_2}^{\circ(k_2)} \circ \ldots
    \circ a_{i_s}^{\circ(k_s)},\quad a_{i_j}\in A,\ k_{i_j} \in \mathbb{Z}.
$$
Since $B$ is a homomorphism from $G_B$ to $G$, we conclude that
$$
B(w) = B\big(a_{i_1}^{\circ(k_1)}\big) B\big(a_{i_2}^{\circ(k_2)}\big)
 \ldots B\big(a_{i_s}^{\circ(k_s)}\big) = u_{i_1}^{k_1} u_{i_2}^{k_2} \ldots u_{i_s}^{k_s}.
$$
Hence, we can find the values of $B$ on the group that
is generated by the set $A$ under the product $\circ$.

b) If the set $\langle A \rangle_{B}$ is equal $G$,
then we can find all values $B(g)$, $g \in G$ by the values $B(a_i)$, $a_i \in A$.
\end{proof}

Now we consider a question on construction of RB-operators on an arbitrary group~$G$.
In this case we do not know generating set of $G_B$,
anyway we can take a generating set of~$G$ and define a map $\beta$ on this set.
Let $G = \langle A \rangle$ be a~group generated by a~set $A = \{ a_i \mid i \in I \}$.
Suppose that $U = \{u_i\}_{i \in I}$ is a subset of~$G$ with the same cardinality as~$A$
and the map $\beta \colon A \to U$ is defined by $\beta(a_i) = u_i$, $i \in I$.

Let us construct a~group $\bar{G}_{\beta}$.
To do this, we take a set of letters
$\bar{A} = \{\bar{a_i} \mid i \in I \}$ which does not intersect with~$A$
and has the same cardinality as~$A$.
Let $\langle \bar{A} \rangle$ be a~free group with basis $\bar{A}$.
We denote the operation in this group by~$\circ$.
Any element of $\langle \bar{A} \rangle$ can be uniquely presented
by a~reduced word of the form
$$
w = \bar{a}_{i_1}^{\circ(k_1)} \circ \bar{a}_{i_2}^{\circ(k_2)}
 \circ \ldots \circ \bar{a}_{i_s}^{\circ(k_s)},
$$
where $\bar{a}_{i_j} \in \bar{A}$, $k_{i_j} \in \mathbb{Z} \setminus \{0 \}$.
Define a~map $\bar{\beta} \colon \langle \bar{A} \rangle \to G$ by the rule
$$
\bar{\beta}(w) = \beta(a_{i_1})^{k_1} \beta(a_{i_2})^{k_2} \ldots \beta(a_{i_s})^{k_s}
 = u_{i_1}^{k_1}  u_{i_2}^{k_2}  \ldots u_{i_s}^{k_s}.
$$
It is easy to see that this map is a homomorphism of
$\langle \bar{A} \rangle$ to $\langle U \rangle$.

Define a map $\pi \colon \langle \bar{A} \rangle \to G$ by the formula
$$
\pi(w) = (a_{i_1}\beta(a_{i_1}))^{k_1}
    (a_{i_2}\beta(a_{i_2}))^{k_2} \ldots
    (a_{i_s}\beta(a_{i_s}))^{k_s}
    \beta(a_{i_s})^{-k_s}
    \beta(a_{i_{s-1}})^{-k_{s-1}} \ldots
    \beta(a_{i_1})^{-k_1}.
    $$
The following lemma clarifies the connection between these two maps.

\begin{lemma} \label{l7.3}
Let $w, w' \in \langle \bar{A} \rangle$, then

a) $\pi(w^{\circ({-1})}) = \bar{\beta}(w)^{-1} \pi(w)^{-1} \bar{\beta}(w)$,

b) $\pi(w \circ w' ) =  \pi(w) \bar{\beta}(w) \pi(w')   \bar{\beta}(w)^{-1}$,

c) $\pi((w')^{\circ(-1)} \circ w \circ w')
 = \bar{\beta}(w')^{-1} \pi(w')^{-1} \pi(w)
 \bar{\beta}(w) \pi(w') \bar{\beta}(w)^{-1} \bar{\beta}(w')$.
\end{lemma}

\begin{proof}
a) For the word
$$
w = \bar{a}_{i_1}^{\circ(k_1)} \circ \bar{a}_{i_2}^{\circ(k_2)}
 \circ \ldots \circ \bar{a}_{i_s}^{\circ(k_s)}
$$
its inverse equals
$$
w^{\circ({-1})} = \bar{a}_{i_s}^{\circ(-k_s)}
 \circ \bar{a}_{i_{s-1}}^{\circ(-{k_{s-1}})}
 \circ \ldots \circ \bar{a}_{i_1}^{\circ(-k_1)},
$$
and
$$
\pi(w^{\circ({-1})})
{=} (a_{i_s}\beta(a_{i_s}))^{-k_s}
    (a_{i_{s-1}}\beta(a_{i_{s-1}}))^{-k_{s-1}} \ldots
    (a_{i_1}\beta(a_{i_1}))^{-k_1}
    \beta(a_{i_1})^{k_1}
    \beta(a_{i_{2}})^{k_{2}} \ldots
    \beta(a_{i_s})^{k_s}.
$$
Hence,
$$
\pi(w^{\circ({-1})})
 = \beta(a_{i_s})^{-k_s} \ldots \beta(a_{i_1})^{-k_1} \pi(w)^{-1}
 \beta(a_{i_1})^{k_1} \ldots \beta(a_{i_s})^{k_s}.
$$

b) Given a word
$$
w' = \bar{a}_{j_1}^{\circ(l_1)} \circ \bar{a}_{j_2}^{\circ(l_2)}
   \circ \ldots \circ \bar{a}_{j_t}^{\circ(l_t)},
$$
we compute
$$
w \circ w'
 = \bar{a}_{i_1}^{\circ(k_1)}  \circ \ldots \circ \bar{a}_{i_s}^{\circ(k_s)}
  \circ \bar{a}_{j_1}^{\circ(l_1)}  \circ \ldots \circ \bar{a}_{j_t}^{\circ(l_t)}
$$
and
\begin{multline*}
\pi(w \circ w')
 = (a_{i_1}\beta(a_{i_1}))^{k_1} \ldots (a_{i_s}\beta(a_{i_s}))^{k_s}
 \cdot (a_{j_1}\beta(a_{j_1}))^{l_1} \ldots  (a_{j_t}\beta(a_{j_t}))^{l_t} \\
 \cdot \beta(a_{j_l})^{-l_t} \ldots \beta(a_{j_1})^{-l_1}
 \cdot \beta(a_{i_s})^{-k_s} \ldots \beta(a_{i_1})^{-k_1}.
\end{multline*}
Hence,
$$
\pi(w \circ w' )
 = \pi(w) \beta(a_{i_1})^{k_1} \ldots \beta(a_{i_s})^{k_s}
   \pi(w') \cdot \beta(a_{i_s})^{-k_s} \ldots \beta(a_{i_1})^{-k_1}
 = \pi(w) \bar{\beta}(w) \pi(w') \bar{\beta}(w)^{-1}.
$$

c) Take the product $(w')^{\circ(-1)} \circ w \circ w'$ and find
\begin{multline*}
\pi((w')^{\circ(-1)} \circ w \circ w')
 = \beta(a_{j_l})^{-l_t} \ldots \beta(a_{j_1})^{-l_1}
 ( \pi(w')^{-1} \pi(w) \beta(a_{i_1})^{k_1} \ldots \beta(a_{i_s})^{k_s} \pi(w') \\
 \cdot \beta(a_{i_s})^{-k_s} \ldots \beta(a_{i_1})^{-k_1} )
 \beta(a_{j_1})^{l_1} \ldots \beta(a_{j_t})^{l_t} \\
 = \bar{\beta}(w')^{-1} \pi(w')^{-1} \pi(w) \bar{\beta}(w)
 \pi(w') \bar{\beta}(w)^{-1} \bar{\beta}(w'). \qedhere
\end{multline*}
\end{proof}

Let $R = \{ w \in \langle \bar{A} \rangle \mid \pi(w) = e \}$.
Then from the previous lemma it follows

\begin{lemma}
The set~$R$ is a~subgroup of $\langle \bar{A} \rangle$.
Suppose that $\bar{\beta}$ satisfies the following implication
\begin{equation} \label{cond}
\pi(w) = \pi(w') \Rightarrow \bar{\beta}(w) = \bar{\beta}(w'),
\end{equation}
in particular, $\ker(\pi) \subseteq \ker(\bar{\beta})$,
then $R$~is normal in $\langle \bar{A} \rangle$.
\end{lemma}

\begin{lemma}\label{lem:RB-extCond}
Suppose that for some elements $w, w' \in \langle \bar{A} \rangle$
such that $\pi(w) = \pi(w')$, we have $\bar{\beta}(w) \neq \bar{\beta}(w')$,
then $\beta$~has no an extension to an RB-operator on $G$.
\end{lemma}

\begin{proof}
Suppose that $B$~is an extension of~$\beta$ to an RB-operator~$B$ on~$G$ and
$$
w = \bar{a}_{i_1}^{\circ(k_1)} \circ \bar{a}_{i_2}^{\circ(k_2)}
 \circ \ldots \circ \bar{a}_{i_s}^{\circ(k_s)},\quad
w' = \bar{a}_{j_1}^{\circ(l_1)} \circ \bar{a}_{j_2}^{\circ(l_2)}
 \circ \ldots \circ \bar{a}_{j_t}^{\circ(l_t)}.
$$
Consider two elements in $G_B$:
$$
u = a_{i_1}^{\circ(k_1)} \circ a_{i_2}^{\circ(k_2)}
 \circ \ldots \circ a_{i_s}^{\circ(k_s)},\quad
u' = a_{j_1}^{\circ(l_1)} \circ a_{j_2}^{\circ(l_2)}
 \circ \ldots \circ a_{j_t}^{\circ(l_t)}.
$$

Then
$$
B(u) = \beta(a_{i_1})^{k_1}\ldots \beta(a_{i_s})^{k_s}
 = \bar{\beta}(w)
 \neq \bar{\beta}(w')
 = \beta(a_{j_1})^{l_1}\ldots \beta(a_{j_t})^{l_t}
 = B(u').
$$
On the other hand,
$$
B(u)
 = B(\pi(w))
 = B(\pi(w'))
 = B(u'),
$$
a contradiction.
\end{proof}

If $R = \ker(\pi)$ is a~normal subgroup of $\langle \bar{A} \rangle$,
then we can define a~group $\bar{G}_{\beta} = \langle \bar{A} \rangle / R$.
Denote by $[w]$ the element of $\bar{G}_{\beta}$ that is presented by the word~$w$.
The product in $\bar{G}_{\beta}$ is defined by the rule $[w] \circ [w']  = [w \circ w']$.
Then $\bar{G}_{\beta}$ is generated by $[a_i]$, $i\in I$, and is defined by the relations
$[w] = [w']$ for all words
$$
w = \bar{a}_{i_1}^{\circ(k_1)} \circ \bar{a}_{i_2}^{\circ(k_2)}
 \circ \ldots \circ \bar{a}_{i_s}^{\circ(k_s)},\quad
w' = \bar{a}_{j_1}^{\circ(l_1)} \circ \bar{a}_{j_2}^{\circ(l_2)}
 \circ \ldots \circ \bar{a}_{j_t}^{\circ(l_t)}
$$
such that $\pi(w) = \pi(w')$, i.\,e.,
 the following relation holds in  $G$:
\begin{multline*}
(a_{i_1}\beta(a_{i_1}))^{k_1}
    (a_{i_2}\beta(a_{i_2}))^{k_2} \ldots
    (a_{i_s}\beta(a_{i_s}))^{k_s}
    \beta(a_{i_s})^{-k_s}
    \beta(a_{i_{s-1}})^{-k_{s-1}} \ldots
    \beta(a_{i_1})^{-k_1} \\
 =  (a_{j_1}\beta(a_{j_1}))^{l_1}
    (a_{j_2}\beta(a_{j_2}))^{l_2} \ldots
    (a_{j_t}\beta(a_{j_t}))^{l_t}
    \beta(a_{j_l})^{-l_t}
    \beta(a_{j_{t-1}})^{-l_{t-1}} \ldots
    \beta(a_{j_1})^{-l_1}.
\end{multline*}

If condition~\eqref{cond} is satisfied, then the homomorphism
$\bar{\beta} \colon \langle \bar{A} \rangle \to \langle U \rangle \leq G$
induces the homomorphism
$\bar{G}_{\beta} \to \langle U \rangle$, which we denote by the same symbol~$\bar{\beta}$.
Also, the map $\pi \colon \langle \bar{A} \rangle \to G$ induces the map
$\bar{G}_{\beta} \to G$, which we denote by $\bar{\pi}$.
It is easy to see that $\Imm(\pi) = \Imm(\bar{\pi})$.

In the following example we apply Lemma~\ref{lem:RB-extCond}.

\begin{example}
Let
$$
G = \big\langle s_1, s_2 \mid s_1^2 = s_2^2 = e,\,
 s_1 s_2 s_1 = s_2 s_1 s_2 \big\rangle
$$
be the symmetric group $S_3$ on 3~symbols.

a) If we take the map $\beta \colon \{ s_1, s_2 \} \to \{ s_1, s_2 \}$
such that $s_1 \mapsto s_1$, $s_2 \mapsto s_2$, then it has an extension to the RB-operator
$B_{-1}(g) = g^{-1}$ on $S_3$. Let us find the group $\overline{G}_{\beta}$.
Let $\langle \bar{A} \rangle$ be a~free group with basis $\bar{A} = \{ t_1, t_2 \}$.
We have a~map $\pi \colon \langle \bar{A} \rangle \to G$ that is defined by the formula
$$
\pi \big(t_1^{\circ(k_1)} \circ t_2^{\circ(l_1)} \circ
 \ldots \circ t_1^{\circ(k_s)} \circ t_2^{\circ(l_s)}\big)
 = s_2^{-l_s} s_1^{-k_s} \ldots s_2^{-l_1} s_1^{-k_1}.
$$
In this case $R = \{ w \in \langle \bar{A} \rangle \mid \pi (w) = e \}$
is normal in $\langle \bar{A} \rangle$ and the product
in $\overline{G}_{\beta}$ is opposite to the product in~$G$.

b) Note that the group $S_3$ can be also generated by the elements $\tau_1 = s_1$,
$\tau_2 = s_2$, and $\tau_3 = s_1 s_2 s_1$. Put
$$
\beta(\tau_1) = \tau_1,\quad
\beta(\tau_2) = \tau_2,\quad
\beta(\tau_3) = \tau_2 \tau_1
$$
and take $\langle \bar{A} \rangle$ be a~free group with basis $\bar{A} =\{t_1,t_2,t_3\}$.
Then
\begin{gather*}
\pi(t_1 \circ t_1) = \pi(t_2 \circ t_2) = \pi(t_3 \circ t_1) = e, \quad
\pi(t_1 \circ t_2 \circ t_1) = \pi(t_2 \circ t_1 \circ t_2) = \tau_1 \tau_2 \tau_1, \\
\pi(t_1 \circ t_2) = \pi(t_2 \circ t_3) = \pi(t_3 \circ t_3) =  \tau_2 \tau_1, \quad
\pi(t_2 \circ t_1) = \pi(t_3 \circ t_2) = \pi(t_1 \circ t_3) = \tau_1 \tau_2.
\end{gather*}
Hence, $\pi(\langle \bar{A} \rangle)  = S_3$.

On the other hand,
$$
e = \bar{\beta}(t_1 \circ t_1)  \not= \bar{\beta}(t_3 \circ t_1) = t_2.
$$
It means that the condition of Lemma~\ref{lem:RB-extCond}
does not hold and we can not extend the map~$\beta$ to a~Rota---Baxter operator on $S_3$.
\end{example}

Now we can formulate the main result of the present section.

\begin{theorem} \label{thm:RBextension}
Let $G = \langle A \rangle$ be a~group generated by a~set $A = \{ a_i \mid i \in I \}$.
Suppose that $U = \{u_i\}_{i \in I}$ is a subset of~$G$ with the same cardinality as~$A$ and
$\beta \colon A \to U$, $\beta(a_i) = u_i$, $i \in I$.
If condition~\eqref{cond} is satisfied and the map
$\bar{\pi} \colon \bar{G}_{\beta} \to G$ is bijective, then
$\beta$~can be extended to an~RB-operator on~$G$ and $G_B$ is isomorphic to $\bar{G}_{\beta}$.
\end{theorem}

\begin{proof}
Construct the RB-operator~$B$ on~$G$. Let $g \in G$.
Since $\bar{\pi}$ is invertible on~$G$,
there exists a~unique element
$\bar{g} \in \bar{G}_{\beta}$ such that $\bar{\pi}(\bar{g}) = g$.
Suppose that $\bar{g} = [w]$, where
$$
w = \bar{a}_{i_1}^{\circ(k_1)} \circ \bar{a}_{i_2}^{\circ(k_2)}
 \circ \ldots \circ \bar{a}_{i_s}^{\circ(k_s)}
$$
is a~word in the free group $\langle \bar{A} \rangle$. Then we put
$$
B(g)
 = B(\bar{\pi}(\bar{g}))
 = B(\pi(w))
 = \bar{\beta}(w)
 =  u_{i_1}^{k_1} u_{i_2}^{k_2} \ldots  u_{i_s}^{k_s}.
$$
By condition~\eqref{cond} this map is defined on $[w]$ and
does not depend on the choice of~$w$.

Let us prove that~$B$ is an~RB-operator on~$G$.
For $g' \in G$ there exists a~word
$$
w' = \bar{a}_{j_1}^{\circ(l_1)} \circ \bar{a}_{j_2}^{\circ(l_2)}
 \circ \ldots \circ \bar{a}_{j_t}^{\circ(l_t)}
$$
such that $\pi(w') = \bar{\pi}([w']) = g'$. Then
$B(g') = u_{j_1}^{l_1} u_{j_2}^{l_2} \ldots  u_{j_t}^{l_t}$.

On the other hand, by Lemma \ref{l7.3}b
$$
\pi(w \circ w') = \pi(w) \bar{\beta}(w) \pi(w') \bar{\beta}(w)^{-1}
$$
and we have
\begin{multline*}
B( gB(g)g'B(g)^{-1})
 = B(\pi(w \circ w') )
 = \bar{\beta}(w \circ w')
 = \bar{\beta}(w) \bar{\beta}(w') \\
 = u_{i_1}^{k_1} u_{i_2}^{k_2} \ldots  u_{i_s}^{k_s}
 \cdot u_{j_1}^{l_1} u_{j_2}^{l_2} \ldots  u_{j_t}^{l_t}
 = B(g)B(g').
\end{multline*}
Hence, $B$ is indeed an RB-operator on $G$.
\end{proof}

The next proposition shows that the condition~\eqref{cond}
is not sufficient for existence of an~RB-operator extending given map~$\beta$.

\begin{proposition}
Consider the symmetric group $S_3$. Then

a) the map $\beta \colon s_1 \mapsto s_1$, $s_2 \mapsto e$ defines the group
$\bar{G}_{\beta}$  and this group is isomorphic to $\mathbb{Z}_2\times \mathbb{Z}_2$.

b) There is no an~RB-operator on $S_3$ extending~$\beta$.
\end{proposition}

\begin{proof}
a) Take a free group $\langle t_1, t_2 \rangle$. Then
we can find
$$
\pi(t_1 \circ t_1) = \pi(t_2\circ t_2) = e,\quad
\pi(t_1 \circ t_2) = \pi(t_2\circ t_1) = s_2 s_1.
$$
From equalities
$$
\bar{\beta}(t_1 \circ t_1) = \bar{\beta}(t_2\circ t_2) = e,\quad
\bar{\beta}(t_1 \circ t_2) = \bar{\beta}(t_2\circ t_1) =  s_1
$$
it follows that \eqref{cond} holds and we can define
$$
\bar{G}_{\beta}  = \langle [t_1], [t_2] \mid [t_1] \circ [t_1] = [t_2] \circ [t_2] = e,\,
[t_1] \circ [t_2] = [t_2]\circ [t_1] \rangle \cong \mathbb{Z}_2 \times \mathbb{Z}_2.
$$
Hence, the map~$\beta$ does not define a group operation on the set $S_3$.

b) Suppose that such RB-operator $B$ exists.
Since $B(s_2) = e$, then by Lemma \ref{kernelCosets}
$$
B(s_1 s_2) = B(s_2s_1s_2) = B(s_1 s_2 s_1).
$$
We arrive at a contradiction by the following equalities,
$$
s_1 B(s_1s_2)
 = B(s_1)B(s_1s_2)
 = B(s_1\cdot s_1\cdot s_1s_2\cdot s_1)
 = B(s_1 s_2 s_1)
 = B(s_1s_2). \qedhere
$$
\end{proof}

Now we consider RB-operators on free non-abelian groups.
In this case we do not need to define groups
$\langle \bar{A} \rangle$ and $\bar{G}_{\beta}$.
We can define a~new operation $\circ$ directly on the free group.

The following proposition shows that in some cases
we can reconstruct RB-operators by its values on a basis of a~free group.

\begin{proposition}
Let $F$ be a free group with a basis $A = \{a_i\}_{i \in I}$. Then

a) the map $\beta \colon a_i \mapsto e$ for all $i \in I$
has the unique extension to an RB-operator.
This operator is the trivial one, $B_0(g) = e$, $g \in F$.

b) the map $\beta \colon a_i \mapsto a_i^{-1}$ for all $i \in I$
has the unique extension to an RB-operator.
This operator is $B_{-1}(g) = g^{-1}$, $g \in F$.
\end{proposition}

\begin{proof}
a) Suppose that $B$ is an extension of $\beta$ to an RB-operator on $F$.
Below we use the notation of the operation~$\circ$ in sense of~\eqref{R-product}.
In this case $a_i^{\circ (k)} = a_i^k$, and hence $B(a_i^{\circ (k)}) = e$
for all $i \in I$, $k \in \mathbb{Z}$. If
$$
w = a_{i_1}^{\circ(k_1)} \circ a_{i_2}^{\circ(k_2)} \circ \ldots
 \circ a_{i_s}^{\circ(k_s)},
$$
where $a_{i_j} \in A$, $k_{i_j} \in \mathbb{Z}$, is an arbitrary word, then
$$
B(w)
 = B\big(a_{i_1}^{\circ(k_1)}\big)
   B\big(a_{i_2}^{\circ(k_2)}\big)
   \ldots B\big(a_{i_s}^{\circ(k_s)}\big)
 = e.
$$
Let us show that in this case $G_B$ is also generated by~$A$.
It follows from the equality
\begin{multline*}
w = a_{i_1}^{\circ(k_1)} \circ a_{i_2}^{\circ(k_2)}
  \circ \ldots \circ a_{i_s}^{\circ(k_s)} \\
  = (a_{i_1} B(a_{i_1}))^{k_1} (a_{i_2} B(a_{i_2}))^{k_2}
 \ldots (a_{i_s} B(a_{i_s}))^{k_s}
 B(a_{i_s})^{-k_s} B(a_{i_{s-1}})^{-k_{s-1}} \ldots  B(a_{i_1})^{-k_1} \\
  = a_{i_1}^{k_1} a_{i_2}^{k_2} \ldots  a_{i_s}^{k_s}.
\end{multline*}

b) The proof is similar.
\end{proof}

\begin{question}
Find Rota---Baxter operators on free non-abelian groups.
In particular, is it true that any map~$\beta$
defined on a~basis of a~free group can
be extended to an RB-operator on this group?
\end{question}

The following example shows that the values of~$\beta$
on the generators of $F_n$ do not define values of~$\beta$ on all elements of $F_n$.

\begin{example} \label{e7.8}
Let $F_2 = \langle a, b \rangle$ be the free group of rank 2.
Put $\beta(a) = a$, $\beta(b) = e$.
By Proposition \ref{fr}, there exists an RB-operator on $F_2$
which is extension of the map~$\beta$.
This operator is the homomorphism $B \colon F_2 \to \langle a \rangle$
that is defined on the generators as follows, $B(a) = \beta(a) = a$,
$B(b) = \beta(b) = e$.
Let us denote the product in $F_2$ as $\cdot$ and the product in $(F_2)_B$ as $\circ$.
We try do define the operation~$\circ$ on~$F_2$ applying only the map~$\beta$.
We have
\begin{gather*}
a^{\circ(k)} = a^k,\quad \beta(a^{\circ(k)}) = \beta(a)^k = a^k,\quad k \in \mathbb{Z}, \\
b^{\circ(l)} = b^l,\quad \beta(b^{\circ(l)}) = \beta(b)^l = e,  \quad l \in \mathbb{Z}.
\end{gather*}
For words of syllable length two we have
$$
a^{\circ(k)} \circ b^{\circ(l)} = a^{2k} b^l a^{-k},\quad
b^{\circ(l)} \circ a^{\circ(k)}=  b^l a^{k},\quad k,l \in \mathbb{Z}.
$$
By Proposition~\ref{for} for word of an arbitrary length we have
$$
w = a^{\circ(k_1)} \circ b^{\circ(l_1)} \circ a^{\circ(k_2)} \circ b^{\circ(l_2)}
  \circ \ldots \circ a^{\circ(k_s)} \circ b^{\circ(l_s)}
  = a^{2k_1} b^{l_1} a^{2k_2} b^{l_2}\ldots a^{2k_s} b^{l_s} a^{-\sum k_i}.
$$
We can find
$$
\beta(w) = a^{\sum k_i}.
$$
Hence, we can construct extension of~$\beta$ to the
subgroup~$S = \langle \{a,b\}, \circ \rangle$
generated by~$a$ and~$b$ in $F_B$.
So, $\ker(\beta) = \{ w\mid \sum k_i = 0 \}$
is the normal closure of~$b$ in $S$.

Hence, $S$~is not equal to $F_2$. In particular, $ab$ does not lie in~$S$.

Let us add a new generator $ab$ to the set $\{a, b\}$ and put
$$
\beta_1(a) = \beta_1(ab) = a,\quad
\beta_1(b) = e.
$$
Then in the group $T = \langle \{a, b, ab\},\circ \rangle$ we have by Proposition~\ref{for}
$$
(ab)^{\circ (m)} = (aba)^m a^{-m}, \quad m\in \mathbb{Z}.
$$
In particular, $ab,a^{-1} b^{-1}\in T$, and the group~$T$ consists of elements
\begin{multline*}
 a^{\circ(k_1)} \circ b^{\circ(l_1)} \circ (ab)^{\circ(m_1)}
 \circ a^{\circ(k_2)} \circ b^{\circ(l_2)} \circ  (ab)^{\circ(m_2)} \circ \ldots
  \circ a^{\circ(k_s)} \circ b^{\circ(l_s)} \circ (ab)^{\circ(m_s)} \\
  = a^{2k_1}  b^{l_1} (aba)^{m_1} a^{2k_2} b^{l_2} (aba)^{m_2}
  \ldots  a^{2k_s}  b^{l_s} (aba)^{m_s} a^{-\sum (k_i+m_1)}
\end{multline*}
for integers $k_i, l_i, m_i$. It is easy to check that $T\neq F_2$.
\end{example}

\begin{question}
Is it true that the group $(F_2)_B$ from this example is not finitely generated?
\end{question}

In~Example~\ref{e7.8} we have defined~$\beta$ on the group
$\langle \{a, b\}, \circ \rangle$ that is a subset of $F_2$.
We have seen that $\beta$~has a~non-trivial kernel
and its image is not equal to $F_2$.

\begin{question}
Is there an RB-group $(G, B)$ such that $B$ is a surjective map with a~non-trivial kernel?
\end{question}

\section{Extensions of RB-groups}

If $(G, B)$ is an RB-group, then an RB-group $(H, B_H)$
is called an RB-subgroup of $(G, B)$, if $H$ is a subgroup of $G$
and $B_H$ is the restriction of $B$ to $H$.
We can define a~short exact sequence of RB-groups
\begin{equation} \label{ses}
1 \to (H, B_H) \to (G, B) \to (L, B_L) \to 1,
\end{equation}
where $(H, B_H)$ is an RB-subgroup of $(G,B)$ and there is an epimorphism of RB-groups
$(G, B) \to (L, B_L)$ such that $B_L$ is induced by $B$.
We say that $(G, B)$ is an extension of $(H, B_H)$ by $(L, B_L)$.

As in the case of groups we can formulate
\begin{question}
Let $(H, B_H)$ and $(L, H_L)$ be two RB-groups.
Find all RB-groups $(G, B)$ for which
there exists the short exact sequence~(\ref{ses})?
\end{question}

By Proposition~\ref{direct}, such RB-groups $(G,B)$ exist.

Recall the construction of wreath product (see, for example,~\cite{KM}).
Let $H$ and $L$ be groups, $\Fun(L, H)$ and $\fun(L, H)$
are Cartesian and direct sums of copies $H$ indexed by elements of $L$.
It means that $\Fun(L, H)$ is the group of all functions $L \to H$
and $\fun(L, H)$ is its subgroup of functions with a finite support.
For $f \in \Fun(L, H)$, $l \in L$, the function $f^l$ is defined by the rule
$f^l = f(l x)$, $x \in L$. The map
$$
\hat{l} \colon \Fun(L, H) \to \Fun(L, H),\quad f \mapsto f^l,
$$
is an automorphism of $\Fun(L, H)$, which sends $\fun(L, H)$ to itself, and the maps
$$
L \to \Aut( \Fun(L, H) ),\quad L \to \Aut( \fun(L, H) ),\quad l \mapsto \hat{l},
$$
are isomorphic embeddings. The Cartesian wreath product
$H \bar{\wr} L = L  \cdot \Fun(L, H)$
is a~group with multiplication
$$
l f \cdot l' f' = l l'\cdot f^{l'} f'.
$$
The direct wreath product is the group
$H \wr L = L \cdot \fun(L, H)$.

It is natural to formulate

\begin{question}
Let $G = H \wr L$ be a direct wreath product or $G = H \bar{\wr} L$
be a Cartesian wreath product of $H$ and $L$. What RB-operators can be defined on $G$?
\end{question}

Particular answer on this question gives

\begin{proposition}
Let $G$ be a Cartesian or direct wreath product of groups $H$ and $L$. Then

a) the map $B(lf) = f^{-1}$ defines an RB-operator on $G$;

b) if $L$ is abelian and $\varphi$ is any endomorphism of $L$,
then the map $B(l f) = l^{\varphi}$ defines an RB-operator on $G$;

c) if action of $L$ on $\Fun(L, H)$ (correspondingly, on $\fun(L, H)$)
is trivial and $B_L$ and $B_H$ are RB-operators on $L$ and on $\Fun(L, H)$
(on $\fun(L, H)$ respectively), then the map $B(l f) = B_L(l) B_H(f)$
defines an RB-operator on~$G$.
\end{proposition}

\begin{proof}
a) Since $G$ has exact factorization $G = L \cdot \Fun(L, H)$ or $G = L \cdot \fun(L, H)$,
then the result follows from Example~\ref{exm:split}.

b) Follows from Proposition~\ref{prop:RB-Hom}b.

c) Follows from Proposition~\ref{direct}.
\end{proof}

Non-trivial examples of RB-operators on $\Fun(L, H)$ and
on $\fun(L, H)$ can be found by Theorem~\ref{theo:directProduct}.

It is well-known that any extension of a group~$H$ by a group~$L$
can be embedded into the Cartesian wreath product $H \bar{\wr} L$,
it is the Frobenius embedding.

For extensions of RB-groups we can formulate
\begin{question}
Is it true that any RB-extension of $(H, B_H)$ by $(L, B_L)$
can be embedded into an RB-group $(H \bar{\wr} L, B)$
for some an RB-operator $B$ on $H \bar{\wr} L$?
\end{question}

By analogy with extension and lifting problems for
 automorphism groups of extensions, we can formulate
\begin{question}
Suppose that
$$
1 \to H \to G \to L \to 1
$$
is a short exact sequence of groups.

a) {\it Extension problem}.
Let $(H, B_H)$ be an RB-group. Does there exist an RB-group $(G, B)$
such that $(H, B_H)$ is its RB-subgroup?

b) {\it Lifting problem}.
Let $(L, B_L)$ be an RB-group. Do there exist an RB-group $(G, B)$
and an epimorphism $(G, B)\to (L, B_L)$ such that $B_L$ is induced by $B$?
\end{question}

\section{Connection with associated Lie ring}

The main motivation of the authors L. Guo et al.~\cite{Guo2020}
to introduce Rota---Baxter operators on groups was the connection
with Rota---Baxter operators on Lie algebras, when the initial group is Lie one.
Let us consider another known construction of Lie ring by a given (not necessary Lie) group.

Let $G$ be a group, consider its lower central series
$G_i$, $i=1,2,\ldots$, where
$G_1 = G$ and $G_{i+1} = [G,G_i]$.
The associated graded abelian group
$L(G) = \bigoplus_{n\geq1}G_n/G_{n+1}$
has the structure of Lie ring under the product
$[xG_{i+1},yG_{j+1}] = [x,y]G_{i+j+1}$, where $[x,y]$ is usual commutator.

\begin{proposition}\label{prop:LieRing}
Let $G$ be a group and let $B$ be a Rota---Baxter operator on $G$ such that
$B(G_n)\subset G_n$ for all $n\geq1$.
Then $R$ defined as $R(xG_{i+1}) = B(x)G_{i+1}$ is a Rota---Baxter operator
of weight~1 on the Lie ring $L(G)$.
\end{proposition}

\begin{proof}
Given $h\in G_i$ and $g\in G_{i+1}$,
it is easy to check by Lemma~\ref{lem:elementary}d that
$B(h)^{-1}B(hg)\in G_{i+1}$, so $R$ is well-defined.
Let $x\in G_i$ and $y\in G_j$. We have to state the equality
$$
[R(xG_{i+1}),R(yG_{j+1})]
 = R( [R(xG_{i+1}),yG_{j+1}] + [xG_{i+1},R(yG_{j+1})] + [xG_{i+1},yG_{j+1}] ),
$$
which equals in terms of $B$ to the following one,
\begin{equation}\label{LieRing}
B(x)^{-1}B(y)^{-1}B(x)B(y)G_{i+j+1}
 = B( [B(x),y] + [x,B(y)] + [x,y] )G_{i+j+1}.
\end{equation}
Applying the commutator properties and the fact that $G_{i+j}/G_{i+j+1}$ is abelian,
we may rewrite the inner expression inside the brackets
from the right-hand side of~\eqref{LieRing} modulo $G_{i+j+1}$ as follows,
\begin{multline*}
[B(x),y] + [x,B(y)] + [x,y] \\
 = \underline{ \underline{ [B(x),y]^{B(y)} }} + \underline{[x,B(y)]^{B(x)} + [x,y]^{B(y)B(x)} }
 + \underline{ \underline{ [B(x),B(y)] }} + [B(y),B(x)] \\
 = \underline{[x,yB(y)]^{B(x)} } + \underline{ \underline{ [B(x),yB(y)] } } + [B(y),B(x)]
 = [xB(x),yB(y)] + [B(y),B(x)].
\end{multline*}
The formula~\eqref{LieRing} is a consequence of the equalities fulfilled in $G$
\begin{multline*}
B(y)B(x)B([xB(x),yB(y)] \cdot [B(y),B(x)]) \\
 = B(yB(y)xB(y)^{-1})B([xB(x),yB(y)] \cdot [B(y),B(x)]) \\
 = B\big( yB(y)x\underline{B(y)^{-1}B(y)}
 B(x)(xB(x))^{-1}(yB(y))^{-1}xB(x)y\underline{B(y)} \\
 \times \underline{B(y)^{-1}}B(x)^{-1}B(y)\underline{B(x)B(x)^{-1}}B(y)^{-1} \big) \\
 = B( yB(y)\underline{xB(x)B(x)^{-1}x^{-1}}
   B(y)^{-1}y^{-1}xB(x)yB(x)^{-1}\underline{B(y)B(y)^{-1}}) \\
 = B( \underline{yB(y)B(y)^{-1}y^{-1}}xB(x)yB(x)^{-1} )
 = B(x)B(y). \qedhere
\end{multline*}
\end{proof}

\begin{example}
Let $G$ be a 2-step nilpotent group and $B$ is an RB-operator on~$G$
defined as in Proposition~\ref{prop:center-conj}.
Then by Proposition~\ref{prop:LieRing} we get an RB-operator
$R\colon L(G)\to L(G)$, where $L(G) = G/[G,G]\oplus [G,G]$.
It is easy to verify that $R = -\id$.
\end{example}

\section{Rota---Baxter operators on finite simple groups}

Given a group $G$, an automorphism~$\psi$ of $G$ is called fixed-point-free, if
$\psi(g) = g$ implies $g = e$.
Let us state a group analogue of Proposition 2.21 from~\cite{BurdeGubarev} proven for Lie algebras
(see also~\cite{BelaDrin82}).

\begin{theorem}\label{thm:invertibleRB}
Let $G$ be a finite non-abelian simple group.
Let $B$~be a Rota---Baxter operator on~$G$ such that $\ker B = \{e\}$.
Then $B(g) = g^{-1}$.
\end{theorem}

\begin{proof}
Since $B$ is a homomorphism from $G_B$ to $G$ with trivial kernel,
we conclude that $G_B\cong G$, and $G_B$ is also simple.
Suppose that $\ker(B_+)\neq G_B$, where $B_+(g) = gB(g)$ is a~homomorphism from $G_B$ to $G$.
By simplicity of $G_B$, we obtain that $\ker B_+ = \{e\}$.
So, we have an automorphism $\varphi = B^{-1}B_+$ of the simple group $G_B$.
Note that $\varphi$ is a fixed-point-free. Indeed,
$g = \varphi(g) = B^{-1}(gB(g))$ would imply the equality $B(g) = gB(g)$, i.\,e., $g = e$.
By~\cite{Rowley}, a finite group admitting a fixed-point-free automorphism is solvable.
We have a contradiction.
\end{proof}

\begin{corollary}
Let $G$ be a finite group and let $B$ be an RB-operator on~$G$.
If the group $G_B$ is non-abelian simple, then $B$ is elementary and $G\cong G_B$.
\end{corollary}

\begin{proof}
Since $\ker B$ is a normal subgroup of $G_B$ and $G_B$ is simple,
we have two cases: either $\ker B = G$ or $\ker B = \{e\}$.
If $\ker B = G$, then $B$ is elementary and $G\cong G_B$.
If $\ker B = \{e\}$, then $G\cong G_B$ by Theorem~\ref{thm:invertibleRB}.
\end{proof}

\begin{corollary}\label{coro:invertibleRB}
Let $G$ be a finite non-abelian simple group.
If $G$ is not factorizable, then $G$ is RB-elementary.
\end{corollary}

\begin{proof}
Let $B$ be a nonelementary Rota---Baxter operator on such a group $G$.
By Theorem~\ref{thm:invertibleRB} we have a proper factorization~\eqref{ImageFactorization},
it is a contradiction.
\end{proof}

List of finite simple non-factorizable groups is presented in \cite[Table 4.1]{Kirtland}.
It includes some finite groups of Lie type (classical and exceptional) and 15 sporadic groups.
We are going to describe all Rota---Baxter operators on all sporadic groups.

\begin{table}[h]
\begin{center}
Table: exact factorizations of sporadic groups~\cite{Giudici} \\
\smallskip
\begin{tabular}{c|c|c}
$G$ & $A$ & $B$ \\
\hline
$M_{11}$ & $M_{10}$ & $\mathbb{Z}_{11}$ \\
         & $M_{9}.2$ & $\mathbb{Z}_{11}\rtimes\mathbb{Z}_5$ \\
$M_{12}$ & $M_{11}$ & $A_4$, $D_{12}$, $\mathbb{Z}_{6}\times\mathbb{Z}_2$ \\
         & $M_{9}.2$ & $\mathrm{PSL}_2(11)$ \\
$M_{23}$ & $M_{22}$ & $\mathbb{Z}_{23}$ \\
         & $\mathbb{Z}_{23}\rtimes\mathbb{Z}_{11}$ & $\mathrm{P\Sigma L}_3(4)$, $2^4\rtimes A_7$ \\
$M_{24}$ & $M_{23}$ & $A_4\times\mathbb{Z}_2$, $D_{24}$, $D_8\rtimes \mathbb{Z}_3$ \\
         & $M_{23}$ & $S_4$ \\
         & $M_{23}$ & $B_{24}$ \\
         & $\mathrm{PSL}_2(23)$ & $\mathrm{P\Sigma L}_3(4)$, $2^4\rtimes A_7$
\end{tabular}
\end{center}
\end{table}

\vspace{-0.75cm}

\begin{theorem} \label{thm:sporadic}
Let $G$ be a sporadic simple group.
If $G = M_{22}$ or $G$ is not a Mathieu group, then $G$ has only elementary RB-operators.
If $G\in\{M_{11},M_{12},M_{23},M_{24}\}$, then all RB-operators on $G$ are splitting (see Table).
\end{theorem}

\begin{proof}
Our main tool in the proof is the list of all factorizations of sporadic groups~\cite{Giudici}.
Suppose that $B$ is a nonelementary RB-operator on $G$.
Thus, $\ker B\neq\{e\}$ and $\ker(B_+)\neq\{e\}$ by Corollary~\ref{coro:invertibleRB}.
We have a factorization~\eqref{ImageFactorization}.
Suppose that one of the factors in~\eqref{ImageFactorization} is simple.
We may assume that $\Imm(B)$ is simple, otherwise we consider $\widetilde{B}$ and $\Imm(\widetilde{B})$.
Since $\ker B_+$ is a~normal subgroup of $\Imm(B)$, we conclude that $\ker B_+ = \Imm(B)$.
By~\eqref{FactorIso}, $\Imm B_+ = \ker B$. The last equality due to Proposition~\ref{SplittingCond}
implies that $B$ is splitting.

Suppose initially that a factorization~\eqref{ImageFactorization} is exact.
Since $|G| = |\ker B|\cdot|\Imm B|$, we again get the equality $\ker B_+ = \Imm(B)$,
and so $B$ is splitting.

Now, we look through all factorizations of sporadic groups.
In Tables~1 and~2~\cite{Giudici} we have only one non-exact factorization
with both non-simple factors, it is
$$
He = (\mathrm{Sp}(4,4).2)\cdot(7^2\rtimes SL_2(7)).
$$
Let $\Imm(B) = \mathrm{Sp}(4,4).2$ and $\Imm(B_+) = 7^2\rtimes SL_2(7)$, then
$|\ker B| = |He|/|\Imm(B)| = 2058$.
Since $|\Imm(B_+)| = 8\cdot2058$, we get by~\eqref{FactorIso}
that $\ker(B_+)$ is a~subgroup of $\Imm(B)$ of index~8.
Analyzing maximal subgroups of $\mathrm{Sp}(4,4)$~\cite{Sp44},
we conclude that this case for a~Rota---Baxter operator~$B$ is not possible.
\end{proof}

\section*{Acknowledgments}
Authors are grateful to participants of the seminar ``\'{E}variste Galois'' at
Novosibirsk State University for fruitful discussions.
The second author is grateful to Alexey Staroletov for the helpful discussions.

Vsevolod Gubarev is supported by the Program of fundamental scientific researches of
the Siberian Branch of Russian Academy of Sciences, I.1.1, project 0314-2019-0001.

Valery G. Bardakov is supported by Ministry of Science and Higher Education of Russia
(agreement No. 075-02-2020-1479/1).

The results of~\S3,~\S5,~\S7, and~\S8 are supported by Ministry of Science
and Higher Education of Russia (agreement No. 075-02-2020-1479/1),
while the results of~\S4,~\S6,~\S9, and~\S10 are supported by the Program
of fundamental scientific researches of
the Siberian Branch of Russian Academy of Sciences, I.1.1, project 0314-2019-0001.

\medskip

\noindent Valeriy G. Bardakov \\
Sobolev Institute of Mathematics \\
Acad. Koptyug ave. 4, 630090 Novosibirsk, Russia \\
Novosibirsk State University \\
Pirogova str. 2, 630090 Novosibirsk, Russia \\
Novosibirsk State Agrarian University \\
Dobrolyubova str., 160, 630039 Novosibirsk, Russia \\
Regional Scientific and Educational Mathematical Center of Tomsk State University \\
Lenin ave. 36, 634009 Tomsk, Russia \\
email: bardakov@math.nsc.ru

\medskip
\noindent Vsevolod Gubarev \\
Sobolev Institute of Mathematics \\
Novosibirsk State University \\
e-mail: wsewolod89@gmail.com

\end{document}